\documentclass{amsart}
\usepackage{latexsym,amsxtra,amscd,ifthen}
\usepackage{amsfonts}
\usepackage{verbatim}
\usepackage{amsmath}
\usepackage{amsthm}
\usepackage{amssymb}

\numberwithin{equation}{section}

\usepackage[dvips]{graphicx}     
\usepackage{array,wrapfig}
\theoremstyle{plain}
\newtheorem{theorem}{Theorem}[section]

\newtheorem{prop}[theorem]{Proposition}
\newtheorem{lemma}[theorem]{Lemma}
\newtheorem{cor}[theorem]{Corollary}

\newtheorem{conj}[theorem]{Conjecture}
\newtheorem{quest_as_problem}[theorem]{Question}

\theoremstyle{definition}

\newtheorem{rema}[theorem]{Remark}

\newcommand{\F}{\ensuremath{\mathbb F}}
\newcommand{\KK}{\ensuremath{\mathbf k}}  

\newcommand{\Q}{\ensuremath{\mathbb Q}}

\newcommand{\Z}{\ensuremath{\mathbb Z}}

\def\char{\mbox{char\,}}

\def\deg{\mbox{deg\,}}
\def\dim{\ensuremath{\mbox{dim\,}}}

\newcommand{\len}{\mathop{\mathrm{len}}\nolimits}
\newcommand{\lcm}{\mathop{\mathrm{lcm}}}

\def\modd{\mbox{mod-}}

\def\ob{\: {\mathcal O}b \:}

\begin{document}

\title {Algebras of linear growth and the dynamical Mordell--Lang conjecture}

\author{Dmitri Piontkovski}

      \address{Department of Mathematics for Economics,
Myasnitskaya str. 20, State University `Higher School of Economics', Moscow 101990, Russia
}
\thanks{The article was prepared within the framework of the Academic Fund Program at the National Research University Higher School of Economics (HSE) in 2017--2018(grant 17-01-0006) and by the Russian Academic Excellence Project ``5--100''.}

\email{dpiontkovski@hse.ru}



\date{\today}

 \begin{abstract}
Ufnarovski remarked in  1990 that it is unknown whether any finitely presented associative algebra of linear growth is automaton, that is, whether 
the set of normal words in the algebra form a regular language. If the algebra is graded, then the rationality of the Hilbert series of the algebra follows from the affirmative answer to Ufnarovski's question. Assuming that the ground field has a positive characteristic, we show that the answer to Ufnarovskii's question is positive if and only if the basic field is an algebraic extension of its prime subfield. 
Moreover, in the "only if" part we show that there exists a finitely presented graded algebra of linear growth with irrational Hilbert series. 
In addition, over  an arbitrary infinite basic field, 
the set of Hilbert series of the quadratic algebras of linear growth with $5$ generators is infinite. 

Our approach is based on a connection with the dynamical Mordell--Lang conjecture. This conjecture describes the intersection of an orbit of an algebraic variety endomorphism with a subvariety. We show that the positive answer to Ufnarovski's question implies some known cases of the dynamical Mordell--Lang conjecture. In particular, the positive answer for a class of algebras is equivalent to the Skolem--Mahler--Lech theorem which says that the set of the zero elements of any linear recurrent sequence over a zero characteristic field is the finite union of several arithmetic progressions. In particular, the counter-examples to this theorem in the finite characteristic case give examples of algebras with irrational Hilbert series. 
 \end{abstract}

\maketitle


\section{Introduction}

Let $\KK$ be a field. 
We consider $\Z$-graded associative $\KK$-algebras of the form 
$A =A_0\oplus A_1\oplus \dots$ where 
$A_0 = \KK$.
 All our algebras are assumed to be finitely generated (in other words, affine), 
so that   $\dim A_n <\infty$ for all $n$ and the Hilbert function
 $h_A: n \mapsto \dim A_n$ is well-defined.  
 This function is bounded (that is, there is $C>0$ such that  $h_A (n) < C$ for all $n\ge 0$)
  if and only if the Gelfand--Kirillov dimension of $A$ is at most one. In this case we call $A$ an algebra of linear growth. 

Suppose that a graded algebra $A$ is finitely presented. 
In 1978, Govorov had proved that if the defining relations of $A$  are monomials in generators,
then the Hilbert function $h_A(n)$ is a linear recurrence or,  equivalently, 
the 
Hilbert series $H_A(z) = \sum_{n\ge 0} h_A(n) z^n$ is a rational function. 
He had also conjectured that this holds for each finitely presented graded algebra. Whereas  
the conjecture had occurred to be wrong in general (the first counter-examples 
were constructed by Ufnarovski and Kobayashi, see~\cite{ufn}), 
for many classes of algebras the conjecture is still open. In particular,
we do not know whether there exists a Koszul algebra or a Noetherian algebra  with irrational Hilbert series.

Suppose that a graded algebra $A$ has linear growth. Then the range of the function $h_A$ is finite, 
so  the rationality of $H_A(z)$ means  $h_A$ is eventually periodic. Since affine algebras of linear growth are PI~\cite{ssw}, it
follows that $h_A$ is eventually periodic provided that $A$ is either Noetherian or prime.
 If an algebra of linear growth is Koszul~\cite{pp} or satisfies the condition $FP^3$~\cite{pio06} then $h_A$ is eventually periodic as well. 
It follows that $h_A$ is eventually periodic if the algebra $A$ of linear growth is either coherent, or satisfies ACC for two-sided ideals, or is semiprime, see~\cite{pio06}. 

\begin{quest_as_problem}[Govorov conjecture for algebras of linear growth]
\label{govorov_conj}
Does each finitely presented graded $\KK$-algebra of linear growth have rational Hilbert series?
\end{quest_as_problem}

Moreover, there is even a more general conjecture. Recall that a finitely generated algebra is called automaton if the set of its normal words (with 
respect to some admissible ordering of monomials) forms a regular language or, equivalently, is defined by a finite automaton (see~\cite{ufn}).
Ufnarovski~\cite[5.10]{ufn} has referred the following question as opened: Is every finitely presented algebra of linear growth automaton? He himself has noted that the affirmative answer is much more plausible. 

By a classical theorem on the regular languages~\cite{Ch_Sh}, the Hilbert series of an automaton algebra 
 is a Taylor series expansion of a rational function. Hence, the affirmative answer to the following case of Ufnarovski's question 
implies the  affirmative answer to Govorov's Question~\ref{govorov_conj}.


\begin{quest_as_problem}[Ufnarovski question for graded algebras]
\label{ufn_probl}
Is each finitely presented graded $\KK$-algebra of linear growth automaton with respect to (some) 
 degree-lexicographical order on monomials?
\end{quest_as_problem}

We give a solution both of the above problems in the case of the field $\KK$ of finite characteristic. 

\begin{theorem}
\label{th:intro_char_finite}
Suppose that the field $\KK$ has a finite characteristic. Then the answer to each of the above Questions~\ref{govorov_conj} and~\ref{ufn_probl} is affirmative 
if and only if $\KK$ is an algebraic extension of its prime subfield.   
\end{theorem}

Note that the set of Hilbert series of finitely presented algebras for each fixed list of degrees of the generators and the relations of the algebra is finite provided that the basic field $\KK$ is finite. It follows that for an algebra $A$ of linear growth with given quantity  and the  degrees of  generators and  relations over a finite field, both the length of the period and the length of the initial non-periodic segment of the Hilbert function $h_A$ can  take only a finite number of values.
The next theorem shows that this result does not hold for algebras over infinite fields even in the class of homogeneous quadratic algebras.

\begin{theorem}[Theorem~\ref{th:periods_and_segments}]
\label{th:intro_periods_and_segments}
Let $g\ge 5$ be an integer. 
If the field $\KK$ is infinite, then for each $d\ge 1$ there exists a $g$-generated quadratic $\KK$-algebra of linear growth with periodic Hilbert function $h_A(n)$ such that the initial non-periodic segment of the sequence $h_A$ has length $d$.
If, in addition, $\KK$ contains all primitive roots of unity, then the period $T$ of the above Hilbert function can be arbitrary large. 
\end{theorem}

It was shown by Anick that the set $HQ(g,\KK)$ of Hilbert series of quadratic $\KK$--algebras with $g$ generators is infinite in the case $g\ge 7$~\cite[Example~7]{an4}. Moreover, it follows from results of Roos that the set $HQ(6,\KK)$  is infinite provided $\KK$ has zero characteristic 
(since
the Hilbert series of the quadratic duals 
${\widetilde S}^!$ and $S^!$ to the 6-generated commutative algebras described in~\cite[Theorems~1 and~1']{roos} depend on a natural parameter $\alpha$; 
for example, $H_{S^!}(z)^{-1} = 1-6z+8z^2-z^{\alpha+2}$). 
Moreover, Iyudu and Shkarin~\cite[Example~3.1]{iyudu2017automaton} have shown very recently that 
the set $HQ(3,\KK)$ is infinite provided that $\KK$ contains the primitive roots of unity of arbitrary high order. In addition, 
Theorem~\ref{th:intro_periods_and_segments} shows that for arbitrary infinite field $\KK$
the set $HQ(5,\KK)$ is infinite even if we restrict ourselves to the case of algebras of linear growth. 


\medskip


Our approach is based on (a version of) the dynamical Mordell--Lang conjecture~\cite[Conjecture~1.7]{gt09}.

\begin{conj}[The dynamical Mordell--Lang conjecture]
\label{conj:ML_general}
Let ${\mathcal V}$ be a quasiprojective variety defined over a field $\KK$ (of characteristic zero), let $\Phi: {\mathcal V}\to {\mathcal V}$ be any morphism, and let
$\alpha \in {\mathcal V}(\KK)$. Then for each subvariety $Y\subset {\mathcal V}$, the intersection $Y(\KK) \cap {\mathcal O}_\Phi(\alpha)$ is a union of at most finitely many
orbits of the form ${\mathcal O}_{\Phi^k} (\Phi^l(\alpha))$, for   some nonnegative integers $k$ and $l$.
\end{conj}

A particular case of this conjecture is the following classical theorem.

\begin{theorem}[Skolem--Mahler--Lech]
Let $a_n = \sum_{j=1}^d a_{n-j}$ be a linear recurrent sequence of elements of a field of zero characteristic, where $n\ge d$. Then the set of $m$
such that $a_m=0$ is the union of a finite set and a finite collection of arithmetic progressions. 
\end{theorem}

We refer to a set of nonnegative integers which is the  union of a finite set and a finite collection of arithmetic progressions as Skolem--Mahler--Lech set (or SML set for short).
In these terms,  Conjecture~\ref{conj:ML_general} claims that {\em the set $\{n\ge 0 | \Phi^n(\alpha) \in Y(\KK) \}$ is SML}.

Bell has proved the dynamical Mordell--Lang conjecture for affine varieties 
over
 arbitary field $\KK$ of zero characteristic~\cite{bell}.

\begin{theorem}[Bell's generalized Skolem--Mahler--Lech theorem]
\label{th:bell}
Suppose that $\char \KK = 0$. Let ${\mathcal V}$ be an affine variety over $\KK$, let $\alpha$ be a point in $X$, and let $\Phi$ be an automorphism of ${\mathcal V}$. If $Y$ is a subvariety of ${\mathcal V}$, then the set $\{ m\in \Z | \Phi^m \in Y\}$ is a finite union of a finite set and a finite union of two-sided arithmetic progressions. 
\end{theorem}

We use this theorem to establish Ufnarovski conjecture for a class of algebras. Namely, in Section~\ref{sec:VLRs} below we 
associate a quadratic algebra $A = A (V,L,R, \sigma)$ of linear growth to each  quadruple $(V,L,R, \sigma)$, where  
$V$ is a finite-dimensional vector space, $L$ and $R$ are two its subspaces, and 
$\sigma $ is a linear endomorphism of $V$. We show that the Hilbert function $h_A(n)$ is periodic if and only if for each 
$d\ge 0$ the set of natural $m$ such that $\dim \left( R \cap \sigma^{n} L \right)  = d$ is SML (see Corollary~\ref{cor:SML_for_our_algebra}).
It follows that the Govorov conjecture for our class of algebras is equivalent to the dynamical Mordell--Lang conjecture for some class of affine varieties and automorphisms. Fortunately, these classes are covered by Bell's theorem. Therefore, if the field $\KK$ has zero characteristic then the Govorov conjecture holds for all algebras  $A (V,L,R, \sigma)$. Moreover, we show that all such algebras over a filed of zero characteristic are automaton, that is, the Ufnarovski problem also has affirmative answer in this case. On the other hand, we see that for each linear recurrent sequence $\{a_n\}$ there exist an algebra of the form 
$A = A (V,L,R, \sigma)$ such that its Hilbert function rationally depends on 
the sequence $\{c_n \}$ where $c_n = 0$ if $a_n=0$ and $c_n = 1$ otherwise,
see Section~\ref{sec:examples}. This gives the examples mentioned in Theorem~\ref{th:intro_periods_and_segments}. Moreover, the well-known counter-examples 
to the Skolem--Mahler--Lech theorem in positive characteristic give the "only if" part of Theorem~\ref{th:intro_char_finite}.

Note that the above connection of graded algebras with linear recurrences give re-formulation of classical results and problems about linear recurrences in the language of graded algebras. 
In particular, the Skolem problem asks whether there exists algorithm to determine in a finite number of steps if a given (rational) linear recurrent sequence has a term $a_n = 0$. For the algebras  $A = A (V,L,R, \sigma)$ corresponding to recurrent sequences as above, this problem is equivalent to  either of the following three questions: given an integer $c$, does there exist an algorithm deciding whether there is $n>2$ such that $h_n (A) = c$? $h_n (A) \ne c$? $h_n (A) > c$? 

\begin{quest_as_problem}[Generalized Skolem problem]
Does there exist an algorithm which, given finite lists of generators  and relations of a graded algebra $A$ (say, over rational numbers) and two  constants $c$ and $n_0$, always decide whether there exists a positive integer $n\ge n_0$ such that $h_A(n) >c$? 
\end{quest_as_problem}

Note that there is an integer recurrent sequence $\{a_n\}$ such that there does not exist an algorithm to decide, given the same datum, whether there exists $n\ge n_0$ 
such that $h_A(n) >a_n$~\cite{an2}.

We can consider both Problems~\ref{govorov_conj} and~\ref{ufn_probl} as particular cases of 
a general version of the dynamical Mordell--Lang conjecture. For example, suppose that the degrees of the  generators and relations a graded algebra $A$ of linear growth are bounded above by some integer $D$. Consider the category $\modd A^{\le D}$ of right graded $A$-modules whose generators and relation have degrees at least 0 and at most $D$. Let $F$ be the truncation endofunctor of $\modd A^{\le D}$ sending $M$ to $M_{\ge 1}[1]$. For $m>0$,
let  $V_m \subset \ob \modd A^{\le D}$  be the class of modules  $M$ such that
$\dim M_0 = m $. 
In these terms, Govorov's Problem~\ref{govorov_conj}  has the following form close to the dynamical Mordell--Lang Conjecture: {\em Given a finitely presented graded algebra $A$ of linear growth and a positive integer $m$,  is the set $\{n\ge 0 | F^n(A) \in V_m \}$ SML?}

Similarly, let $X= U\cup V\cup W$   be the set of generators of an algebra $A$ of linear growth 
described in Proposition~\ref{pr:mon_basis}. For $x\in X$, denote $\delta_x= \deg x$. 
Let $Z\subset \ob \modd A^{\le D} $  (where we assume $\delta_x \le D$ for all $x\in X$)
 be the class of modules $M$ such that there is a set of generators of $M$ marked by the monomials $w = ac^nb$ with $a\in U, b\in V, c\in W, n\ge 0$ such that  for some  number $N_M>0$ independent on $w$ 
 we have $\deg w =d_a +d_b+nd_c - N_M \le D$. We fix such a set of generators for each $M\in Z$.
For $a\in U, b\in V, c\in W$, let $Y_{abc}\subset Z$
 consists of the modules $M$ defined by the following  (algebraic) condition: if there exists $n$ such that  $d_a +d_b+nd_c - N_M  = 0$, then the generator $w = ac^nb$ is not a linear combination of less generators (in the sense of the ordering mentioned in Proposition~\ref{pr:mon_basis}). Then  Ufnarovski's Problem~\ref{ufn_probl} is equivalent to the following question:
 {\em Given $a\in U, b\in V, c\in W$, is the set $\{n\ge 0 | F^n(A) \in Y_{abc} \}$ SML?}


\medskip

The paper is organized as follows. In Section~\ref{sec:mon_basis}, we deal with general (non-graded) affine algebras of linear growth. 
We recall the monomial bases for such algebras and give a criterion of automaton algebras in terms of this basis. 
Then we show in Corollary~\ref{cor:aut_algebra_basis} for an automaton algebra $A$ there is a finite subset $S\subset A$ which generates $A$ and a weight-lexicographical order on the monomials on $S$ such that 
for some subset $Q\subset \overline S\times \overline S \times \overline S$ 
(where  $\overline S $ is the union of $S$ and the singleton of the empty word) the set of all normal words in $A$  is  
$$
               \{ a c^n b | n\ge 0 , (a,b,c) \in Q \} .
$$
In Section~\ref{sec:Fp}, we show that if the basic field $\KK$ is an algebraic extension of its finite subfield then any graded $\KK$-algebra 
of linear growth is automaton. This gives the ``if'' part of Theorem~\ref{th:intro_char_finite}.  
In the next Section~\ref{sec:VLRs} we introduce and study the algebras $A (V,L,R, \sigma)$. In the next Section~\ref{sec:examples} we clarify their connection with recurrent sequences and consider the examples which give Theorem~\ref{th:intro_periods_and_segments} and the ``only if'' part of Theorem~\ref{th:intro_char_finite}. 



\section{Monomial bases and automaton property}
\label{sec:mon_basis}

Let $A$ be an algebra of linear growth generated by a finite set of homogeneous elements $S$. The  following description of its normal words follows from~\cite[Proposition~2.174b]{bbl}.

\begin{prop}
\label{pr:mon_basis}
Assume a degree-lexicographical order on the monomials on the alphabet $S$. Then there are three 
finite sets $U,V,W$ of words on the alphabet $S$ such that each normal word in $A$ has the form 
$$w = ac^nb, \mbox{ where } a\in U, b\in V, c\in W, n\ge 0.
$$ 
Moreover, the sets  $U,V,W$  may be chosen in such a way that  if the length of $w$ is sufficiently large then the word $c$ is uniquely determined by the word $w$. 
\end{prop}

When a  set described in Proposition~\ref{pr:mon_basis} form a regular language?

\begin{prop}
\label{pr:reg_land}
Suppose that each element of a language $L$ over an alphabet $S$ has the form $w = ac^nb$ where $a\in U, b\in V, c\in W$
for some finite subsets $U,V,W$ of $L$.
Then the language $L$ is regular if and only if for each  $a\in U, b\in V, c\in W$ the set 
$$
N_{a,b,c} = \{ n\in \Z_+ | ac^nb \in L\}
$$
 is the  union of a finite set and finitely many (one-sided) arithmetic progressions. 
\end{prop}

\begin{proof}
Suppose that the language $L$ is regular.  Since the intersection of two regular languages is regular~(see \cite{salomaa}), for each $a\in U, b\in W, c\in U$ the set 
$$
L_{a,b,c} = L\cap \{ ac^nb | n \in \Z_+ \} = \{ ac^nb \in L | n \in N_{a,b,c} \}  
$$
is a regular language. Then its generating function 
$$
f(z) = \sum_{w \in L_{a,b,c} } z^{\len w} = z^{p} \sum_{n\in N_{a,b,c}} z^{qn}
$$
is a rational function, where $p = \len a + \len b$ and $q = \len c$~\cite{Ch_Sh}. Hence the formal power series
$$
g(z) = \frac{1}{1-z} - z^{-p}f(z)
$$
is a rational function as well. Note that $ N_{a,b,c}$ is the set of exponents $n$ such that the term $z^n$ is absent in the power series decomposition of $g(z)$.  By the Skolem--Mahler--Lech theorem, this set is the union of a finite set and finitely many one-sided arithmetic progressions.

Reversely, suppose that each set $ N_{a,b,c}$ is the union  of finitely many one-sided arithmetic progressions of the form $p_k = \{\alpha_k + t \beta_k |t\ge 0\}$ (where singletons are considered as arithmetic progressions with $\beta_k=0$). Then the language 
$$
L_k = \{ w = ac^nb |  n\in p_k\} = \{ ac^{\alpha_k + t\beta_k} b | t\ge 0 \} = ac^{\alpha_k} (c^{\beta_k})^* b
$$
is  regular. Since $L$ is the finite union of languages of the form $L_k$ we conclude that $L$ is regular as well. 
\end{proof}

\begin{cor}
\label{cor:aut_algebra_basis}
Suppose that the algebra $A$ is automaton. 
Then there is a finite subset $S\subset A$ which generates $A$ and a weight-lexicographical order on the monomials on $S$ such that 
for some subset $Q$ of $\overline S\times \overline S \times \overline S$  the set of normal words in $A$  is the set
$$
               \{ a c^n b | n\ge 0 , (a,b,c) \in Q \} 
$$
where $\overline S $ is the union of $S$ and the singleton of the empty word. 

If the algebra $A$ is graded and is automaton with respect to a degree-preserving order on monomials, then the set $S$ may be chosen to be homogeneous and the new order on the monomials on $S$ may be chosen to be degree-preserving. 
\end{cor}

\begin{proof}
Let $S'$ be a homogeneous set of generators of $A$ with respect to which $A$ is automaton. Let $L$ be the regular language on the alphabet $S'$ which consists of the normal words   of $A$.
By Proposition~\ref{pr:mon_basis}, there are subsets $U,V,W \in S'$ such that 
$$L = \cup_{ a\in U, b\in V, c\in W} L_{a,b,c}, $$
where $L_{a,b,c} = \{ac^nb | n \in N_{a,b,c} \}$.

 Suppose that a set $N_{a,b,c}$ is infinite. 
 By Proposition~\ref{pr:reg_land}, it is a union of finite set and a finite collection of arithmetic progressions of the form $p_k =\{\alpha_k + n \beta_k | n\ge 0 \}$. Let $q = q(a,b,c)$ be the least  common multiple of all $\beta_k $s, and let $\tilde c =\tilde c(a,b,c)$ denotes the monomial $c^q$. Since each $p_k$ is the union of progressions with common  difference of $q$ we conclude that there are finite sets $U(a,b,c)$ and $V(a,b,c)$ of normal words such that 
 $$
 L_{a,b,c} = \{ \tilde a {\tilde c}^n \tilde b | \tilde a \in U(a,b,c) , \tilde b \in V(a,b,c) , \tilde c =\tilde c(a,b,c), n\ge 0 \}  \cup L^0(a,b,c),
 $$ 
where $L^0(a,b,c)$ is a finite set. If 
for some $a,b,c$  the sets $U(a,b,c)$ and $V(a,b,c)$ are empty, then we put $\tilde c = 1$ (the empty word) and $L^0(a,b,c) = L(a,b,c)$.

Let $\widetilde S$ be the union of all the sets $U(a,b,c) \cup V(a,b,c)\cup \{\tilde c (a,b,c)\} \cup L^0(a,b,c)$
for all $a\in U, b\in V, c\in W$. Then the set $S =\widetilde S \setminus \{ 1\}$ is obviously a set of monomial generators of $A$. In particular, we can assume that $S'\subset S$.

The order $\le $ on the monomials on $S'$  induces a partial order on the monomials on its superset $S$. Let us extend this order up to a total order on the monomials on $S$ such that $w > w_S$
if $w$ is a monomial on $S'$ which is equal (in the free algebra on $S'$)  to a monomial $w_S$ on $S$. Note that if the original order was  degree-preserving  then the extended order is degree-preserving as well if we put $\deg w_S = \deg w$ for $w,w_S$ as above.

Let $Q'$ be the set of all triples 
 $(\tilde a(a,b,c), \tilde b(a,b,c), \tilde c(a,b,c) ) \in S^3$ such that $L_{a,b,c}$ is infinite, 
 and let $L^0$ be a union of all finite sets $L^0(a,b,c)$ considered as monomials on $S$ and the singleton of the empty word. Then the set 
$
            L^0 \cup   \{ \tilde a {\tilde c}^n \tilde b | n\ge 0 , (\tilde a, \tilde b, \tilde c) \in Q' \} 
$ 
 is a monomial basis of $A$. This set has the desired form
$$
 \{ a c^n b | n\ge 0 , (a,b,c) \in Q \} 
$$
for $Q = Q' \cup ( L^0 \times \{ 1 \}\times \{ 1 \}) $, where 
1 
is the empty word.

 Note that all these monomials are irreducible  since the corresponding monomials on $S'$ are irreducible. It follows that these monomials form the set of normal words of A. 
\end{proof}

\section{The case of a field which is algebraic over $\F_p$}
\label{sec:Fp}

In the next theorem we assume that the order on monomials is degree-preserving, that is,  for a graded algebra $A$ 
we fix a homogeneous set $S$ which generates the algebra and an order on the monomials on $S$ 
such that $m<n \Longrightarrow \deg m \le \deg n$ for each monomials $m$ and $n$ on $S$.

\begin{theorem}
\label{th:finite_field}
Suppose that the basic field $\KK$ is an algebraic extension of a finite field. Let $A$ be a finitely presented graded $\KK$-algebra of linear growth. Then $A$ is automaton.
\end{theorem} 

\begin{proof}
Let $k$ be a finite subfield of $\KK$ which contains all coefficients of the relations of $A$ as noncommutative polynomials on $S$. Since the Buchberger algorithm of noncommutative Groebner basis requires linear operations with elements of $k$ only, it follows that all coefficients of the elements of the Groebner basis of the ideal of relations of $A$ 
belong to $k$ as well. So, the associated monomial algebra $\hat A$ to $A$ coincides with the associated monomial algebra 
 to an algebra defined by the same generators and relations as $A$ over a field $k$. Hence we can assume that $\KK=k$
 is a finite field. 

For $n\ge 0$, let 
$A_{\ge n} = A_n\oplus A_{n+1}\oplus \dots$ be the (right) ideal 
generated by the elements of degree $\ge n$ in $A$ 
and let $M^n = A_{\ge n}[n]$ be its degree shift, that is, a graded right module with the components   $(M^n)_t = A_{n+t}$. Suppose that the homogeneous generators and the relations of $A$ have degrees at most $D$ and suppose that $\dim A_n \ge C$ for all $n\ge 0$. 
It is not hard to see (cf.~\cite[Lemma~3.5]{pio06}) that  both each minimal set of homogeneous generators and 
  each minimal set of homogeneous  relations of the module $M^n$ are concentrated in degrees at least $0$ and at most $D$. It follows from  Proposition~\ref{pr:mon_basis} that the right module $A_n$ is generated by the monomials of the form 
$w^t_{a,b,c}(n) = ac^{t+t_0}b$ where for each $a\in U, b\in V, c\in W$ the number $t_0$ is the minimal non-negative integer such that $\deg a+t_0 \deg c +\deg d \ge n$ and $0\le t \le D$. In the module $M_n$, the degree of the element  $w^t_{a,b,c}(n)$ lies in the interval $[0, D']$ where $D'$ is some constant that does not depend on $n,a,b,c,t$ (say, $D' = 
 \max \{ \deg \widetilde a+(D+1) \deg \widetilde c +\deg \widetilde d | \widetilde a \in U, \widetilde b \in V, \widetilde c \in W \}$). 
 
 Let ${\mathcal M}$ be the set of all right graded $A$-modules with generators $\{w^t_{a,b,c} | a\in U, b\in V, c\in W , t \in [0,D] \}$
of degrees from the interval $[0, D']$ with relations of degree at most $D$. Since we assume that the basic field $\KK$ is finite, the set ${\mathcal M}$ is finite as well. On the other hand, each $M^n$ is isomorphic to some  module $M\in {\mathcal M}$ via the isomorphism induced by 
 the map on generators $w^t_{a,b,c}(n) \mapsto w^t_{a,b,c}$. It follows that there are two positive integers $n$ and $m$ such that there is an isomorphism $\phi: M^{n}\to M^{n+m}$ induced by the map on generators $\phi: w^t_{a,b,c}(n) \mapsto w^t_{a,b,c}(n+m)$. Note that $\phi$ sends each word $w = ac^tb$ of degree at least $n$ to some word $\phi(w)$ of the same form 
 $ac^{t'}b$ where $\deg \phi(w) = m+ \deg w $. Because the order on the set of words is the degree-lexicographical one, the ordering of the words of the form $ac^tb$ (where $a\in U, b\in V, c\in W $ and $t>0$) having the same degree $s$ is determined by the triples $(a,b,c)$ only and does not depend on $t$ and $s$. Thus, for all $s\ge n$ the map $\phi$ induces a  one-to-one correspondence between  the normal words  of degree $s$ and the normal words of degree $s+m$. In particular, 
 for each $a\in U, b\in V, c\in W $  and $t$ such that $s:= \deg a + t\deg c +\deg b \ge n$ we have 
 $t\in N_{a,b,c} $ if and only if $t+m' \in N_{a,b,c} $ where $m' = m/\deg c$ is integer. Hence the set $ N_{a,b,c}$
 is the union of a finite set and several arithmetic progressions with common difference of $m'$.
 \end{proof}

\section{A construction of an algebra by a dynamical system}

\label{sec:VLRs}

Let $V$ be an $m$-dimensional vector space with a basis $X = \{x_1, \dots, x_m\}$, let $L$ and $R$ be two subspaces of $V$, and let 
$\sigma $ be a linear operator on $V$. Let $A = A (V,L,R, \sigma)$ be the algebra generated by the set $Y = X \cup \{a,b,c\}$ subject to the set of relations 
$$
XX \cup Ya \cup bY \cup  aL \cup Rb \cup \{x_i c - c \sigma(x_i) | i = 1, \dots, m \} .
$$
We see that $A$ is a homogeneous quadratic algebra with $g = m+3$ generators and $s = m^2+3m+5 +r +l $ relations, where $r = \dim R$ and $l = \dim L$.

\begin{theorem}
\label{th:main_algebra}
For $n\ge 0$, the graded component $A_{n+3}$ is the direct sum 
of the vector subspaces $U_{n+3} = \KK (\{ c^{n+3}, ac^{n+2}, c^{n+2}b, a c^{n+1}b \} \cup c^{n+2}X )$,
$V_{n+3} = \KK c^{n+1} X b$, $W_{n+3} = \KK aX c^{n+1} $, and $T_{n+3} = \KK ac^nXb$. 
The dimensions of these vectors spaces are $m+4$, $m-r$, $m-\dim \sigma^{n+1}L$, and $m-\dim \left( R + \sigma^n L \right)$, respectively.   
\end{theorem}

\begin{proof}
It follows from the defining monomial relations $XX, Ya, bY$ 
and $\{x_i c - c \sigma(x_i) | x_i\in X \}$ that the vector space  $A_{n+3}$ is equal to $U_{n+3} + W_{n+3} + V_{n+3} + T_{n+3} $. Since $A$ is graded by the degrees of each of the variables $a,b$, and $c$ the sum is direct. It follow now from the relations $x_i c - c \sigma(x_i)$  for $x_i\in X$ that 
each monomial $vc^k$ with $v\in V$ is uniquely re-written modulo the relations of $A$ as an element of the vector space $c^kV$ as $c^k\sigma^k(v)$. So, the monomials listed in the definition of $U_{n+3}$ form a basis of $U_{n+3}$. 

Now, let $B$ be another algebra having the same generators as $A$ and the same relations as $A$ but $bY \cup  aL$.
From now, we consider all monomials as elements of $B$. 
In the algebra $B$ each monomial of the form $avc^k$ 
(respectively, $vc^kb$ and $avc^kb$) with $v\in V$ is uniquely re-written modulo 
the relations of $B$ as an element of the $m$-dimensional vector space $ac^kV \subset B$ as $ac^k\sigma^k(v)$ 
(respectively, as  $c^k\sigma^k(v)b \in c^kVb $  and  $ac^k\sigma^k(v)b \in ac^kVb $).
It follows that $V_{n+3}$ is the quotient space of the  analogous subspace 
$c^{n+1} V b \subset B$
by the subspace $ c^{n+1} R b$ so that 
$\dim V_{n+3} = \dim c^{n+1} V b - \dim   c^{n+1} R b= m-r$. By the same manner, 
$W_{n+3}$ is the quotient of the subspace $ac^{n+1}V \subset B$ by $aLc^{n+1} = ac^{n+1}\sigma^{n+1}L$, so that 
$\dim W = m- \dim (\sigma^{n+1}L)$. 
Finally, $T_{n+3}$ is the quotient of the subset $a c^nVb \subset B$  by the sum of the subspaces $aLc^nb = ac^n\sigma^n(L)b$
and $ac^nRb$ of the $m$-dimensional subspace $ac^nVb\subset B$, so that $\dim T_{n+3} = m- \dim (ac^n\sigma^n(L)b + ac^nRb) 
= m- \dim (R+\sigma^nL)$.
 \end{proof}

\begin{cor}
\label{cor:hA_constr}
The Hilbert function $h_A(n) = \dim A_{n}$ is given by $h_A(0) = 1$, $h_A(1) = g$, $h_A(2) = g^2-s = 3m+4-r-l$ and 
$$
h_A(n+3) = 4m+4 -r-\dim \sigma^{n+2}L - \dim \left( R + \sigma^{n} L \right)
$$ 
for $n\ge 0$.

In particular, for $n\ge m$ we have 
$$
h_A(n+3) = 4m+4-2r-2t + c_{n} ,
$$ where $t = \dim (\sigma ^m L)$ and 
$c_{n} = \dim \left( R \cap \sigma^{n} L \right)$. 
\end{cor}

Recall that a set of non-negative integers is SML iff it is either finite or a union of a finite collection of arithmetic progressions and a finite set.

\begin{cor}
\label{cor:SML_for_our_algebra}
The algebra  $A = A (V,L,R, \sigma)$ has periodic Hilbert function if and only if  for each $d\in [0, \dim R]$ the set 
$N_d = \{ n | n\ge 0, \dim \left( R \cap \sigma^{n} L \right)  = d \}$  is SML.
\end{cor}

Let $t$ be the dimension of the vector space $L' =  \sigma^m (L)$. Then the restriction $\sigma_{L'}: L' \to V$ is an inclusion. 
Let $\sigma'$ be any extension of $\sigma_{L'}$ up to a linear endomorphism of $V$.

Let ${\mathcal V}$ be the affine space $\KK^{mt} \cong V^{\otimes t}$. Obviously, $\sigma'$ naturally induces a linear automorphism $\Phi = {\sigma'}^{t}$ of ${\mathcal V}$. Let us consider each element $v\in {\mathcal V}$ as a collection of $t$ vectors $v_1, \dots, v_{t} \in V $. 
For each $h\ge 0 $ consider the affine variety $Y_h\subset {\mathcal V}$ defined by the condition
$\dim(R + \KK\{ v_1, \dots, v_{t}\} ) \le h$. Let $w_1, \dots, w_t$ be a basis of $L'$ and let $\alpha\in {\mathcal V}$ be the corresponding affine point.
Then Corollary~\ref{cor:SML_for_our_algebra} gives the next

\begin{cor}
\label{cor:periodic_VLRsigma_mod_ML}
 The algebra  $A = A (V,L,R, \sigma)$ has periodic Hilbert function if and only if  for each $h\in [0, r+t]$ the dynamical Mordell--Lang conjecture holds
for the morphism $\Phi$, the subvariety $Y_h \subset {\mathcal V}$, and the initial point $\alpha$.
\end{cor}

According to  Bell's Theorem~\ref{th:bell}, Corollary~\ref{cor:periodic_VLRsigma_mod_ML} implies that the algebra $ A (V,L,R, \sigma)$ has periodic Hilbert function provided that $\char \KK = 0$. Now we use Bell's theorem to establish the Ufnarovski conjecture for $ A (V,L,R, \sigma)$. 

\begin{theorem}
\label{th:ufn_for_VLRsigma}
Let us introduce the degree-lexicographical monomial order with $x_1>\dots >x_m>c>b>a$. 

If the field $\KK$ has zero characteristic, then each algebra of the form $A = A (V,L,R, \sigma)$ is automaton.
\end{theorem}

 \begin{proof}
 In the notation of Theorem~\ref{th:main_algebra}, 
 the subset  of normal words in each $A_{n+3}$ (where we assume $n>m$) consists of four parts which are subsets of  $U_{n+3}$,
$V_{n+3}$, $W_{n+3}$, and $T_{n+3}$, respectively.  Let $L_i$ be the union of the $i$-th parts for all $n > m$, where $i=1,2,3,4$. We will prove the theorem if we show that each language $L_i$ is regular.
 It is sufficient to show that each $L_i$ satisfies the conditions of Proposition~\ref{pr:reg_land}.

We have $L_1 =\{  c^{n+3}, ac^{n+2}, c^{n+2}b, a c^{n+1}b , c^{n+2}x_i |n\ge 3, i=1, \dots, m\}$. Obviously,  $L_1$ is a regular language. Next, let $X' = x_{i_1}, \dots, x_{i_{m-r}} $ be the set of normal words in $V$ modulo $R$ (that is, 
we assume that $x_{i_{k+1}} $ does not belong to $R+ \KK\{  x_{i_1}, \dots, x_{i_{k}}\}$ for each $k=1, \dots, m-r$). Then $L_2 = c^{n+1} X' b$ is regular as well. 

Let ${\mathcal V}, \phi, \alpha$ be as in Corollary~\ref{cor:periodic_VLRsigma_mod_ML}. 
  For each subset $Z \subset \{x_1, \dots, x_m\}$, let us define the subsets $P(Z), Q(Z)\subset X$ by the following conditions on $v\in {\mathcal V}$: 
   '$Z$ is linearly dependent modulo $\KK \{ v_1, \dots, v_{t} \}$' and  
   '$Z$ is linearly dependent  modulo $R +\KK \{ v_1, \dots, v_{t} \}$', respectively. 
   Obviously, the both conditions are algebraic, so that the both $P(Z)$ and $Q(Z)$ are affine varieties. By the same arguments as above, Bell's Theorem~\ref{th:bell} implies that  for each $Z$ the sets $N_P(Z) = \{n \ge m|  \Phi^{n-m+1}\alpha \in P(Z)  \}$   and $N_Q(Z) = \{n \ge m|  \Phi^{n-m}\alpha \in Q(Z) \}$ are SML.
 
 On the other hand,   the subset consisting of the words of length $n+3$ in the  language $L_3$ (respectively, $L_4$) consists of the maximal words of the form $a c^{n+1} X$ (resp., 
 $T_{n+3} = ac^nXb$) which are irreducible with respect to $a c^{n+1} L'_n$ (resp., $a c^{n}(L'_n+R)b$), where  
 $L'_n = \sigma^{n+1}L = \sigma'^{n+1-m} L'_n$. So, this set has the form $a c^{n+1} H$ (resp., $a c^{n}Hb$) where $H$ is the complement in $X$ to the set of the leading monomials of the vector subspace $L'\subset V$ (resp., $L'_n+R \subset V$). 
 This means that $H$ is the maximal subset of $X$ satisfying the following property: if $H = \{x_{i_1}, \dots, x_{i_{k}} \}$ with 
$i_1> \dots > i_k$ and $x_t \notin H$, then the set $Z_t = x_t \cup (H\cap \{ x_{t+1}, \dots, x_n\})$ is linearly dependent modulo $L'_n$ 
(respectively, modulo $L'_n + R$), while $H$ itself is independent modulo $L'_n$ (resp.,  $L'_n + R$). 

 We see that for an arbitrary subset $H\subset X$ the set  $a c^{n+1} H$ (resp., $a c^{n}Hb$) form the 
 $n+3$-th graded component of $L_3$ (resp., of $L_4$)
 if and only if $n \in N_P(Z_t)$ (resp.,  $n \in N_Q(Z_t)$) for all $x_t\in X\setminus H$ and $n \notin N_P(H)$ (resp., $N_Q(H)$). 
It follows that  the  language $L_3$ (respectively, $L_4$) consists of the words of the form $a c^{n+1} x_i$ (resp., $a c^{n}x_ib$) where for each 
$i=1, \dots, m$ the set of possible $n$ runs some  SML set. Thus,  Proposition~\ref{pr:reg_land} shows that the both languages $L_3$ and $L_4$
are regular.
 \end{proof}

\section{Examples and counter-examples}
\label{sec:examples}

Our main class of examples is given by the following two particular cases of Theorem~\ref{th:main_algebra}. 

\begin{cor}[Algebra by a linear recurrence]
\label{cor:reccurent}
Suppose that a sequence $\{a_n\}_{n\ge 0}$ 
of elements of $\KK$ is given by a linear recurrence of some order $d$.
Then there exists a  $\KK$-algebra $A$ with $g = d+3$ generators and 
$s = d^2+4d+5$ quadratic homogeneous relations such that for each $n\ge 0$ we have 
$$
h_A(n+3) = \left\{
\begin{array}{ll}
2d+5, & a_n =0,\\
2d+4, & a_n \ne 0.
\end{array}
\right.
$$ 
\end{cor}

\begin{proof}
Let $a_{n} = \sum_{i=1}^d \gamma_i a_{n-i}$ be the recurrent relation for $a_n$. Denote 
$v_n = (a_n, \dots , a_{n+d-1})^T \in \KK^{d}$. Then $v_{n+1} = M v_n$ for $n\ge 0$, where the $d\times d$ matrix
$$
M = \left(
\begin{array}{ccccc}
0 & 1 & 0 &\dots & 0\\
0 & 0 & 1 &\dots & 0\\
\vdots &  & & \ddots & \\
0 & 0 & 0 &\dots & 1\\
\gamma_{d} &  & \dots & & \gamma_{1}\\
\end{array}
\right) 
$$
is invertible.

Put $V = \KK^{d}$, $L = \KK v_0$, and $R = \{ (0, x_2, \dots, x_d)^T | x_2, \dots, x_d\in \KK \}$. 
Let $\sigma$ be the automorphism of $V$ defined by the matrix $M$, and 
let $A$ be the algebra defined by these data by Theorem~\ref{th:main_algebra}.  
In the notations of Theorem~\ref{th:main_algebra} and Corollary~\ref{cor:hA_constr},
we have $m=d$, $r=d-1$, $l=t=1$. Moreover, we have $a_n = 0$ iff $\sigma^n(L) \subset R$.
So, $c_{n} = 1$ if $a_n = 0$ and $c_{n} = 0$
 if $a_n\ne 0$. By Corollary~\ref{cor:hA_constr}, we conclude that $h_A(n+3) = 2d+4 +c_n$ for all $n\ge 0$.
\end{proof}

Note that 
for any two subspaces $R$ of codimension $1$ and $L$ of dimension 1 
in a finite-dimensional vector space $V$ and any automorphism $\sigma$ of $V$ 
there is a linear recurrence $\{a_n \}$  such that
 $a_n = 0$ 
 iff $\sigma^n(L) \subset R$.
 A nice construction of such a linear recurrence is given in \cite[Section~2]{bell}. Let $V = \KK^m$, and let  $M$ be the matrix of $\sigma$. Then for each $u,v\in V$
 the sequence $a_n = u^T M^n v$ is linear recurrent since, by a well-known theorem of
  Sh\"utzenberger, its generating function  $g(z) = \sum_{n\ge 0} a_n z^n $ is 
rational. On the other hand, if we take nonzero vectors $u,v$ such that  $v\in L$ and $(u,r) = 0$ for each $r\in R$ then $a_n = 0$ iff $\sigma^n(v)\in R$. 

Note that the order $d$ of the linear recurrence $\{a_n \}$ is at most $m$, the degree of the denominator $\det(I-zM)$
of the rational expression for $g(z)$.

 So, we get the following more general version of Corollary~\ref{cor:reccurent}.
 
 \begin{prop}[Linear recurrence by an algebra]
 \label{prop:more_general_reccurence}
Suppose that $R$ is a subspace of codimension one and $L$ is a subspace of dimension one in an $m$-dimensional vector space  $V$, and $\sigma$ is a linear automorphism of $V$. Then for $A = A(V,L,R, \sigma)$ we have 
$h_A(n+3) = 2m+4+c_n,  
$
where $c_n$ is either 1 or 0 according to whether $\sigma(L)\subset R$ or not. So, $c_n = 1$
if and only if $a_n=0$ where $\{a_n \}$ is the linear recurrent sequence of order at most $m$ described above.  
  \end{prop}



 The next examples are particular cases of Proposition~\ref{prop:more_general_reccurence}. 
 We note that the direct application of Corollary~\ref{cor:reccurent} usually leads to a less elegant set 
of relations of the algebra $A$ than the application of Proposition~\ref{prop:more_general_reccurence}.

\begin{prop}[Fermat algebras]
For any two parameters $\alpha, \beta \in \KK^{\times}$ consider the algebra $A= A_{\alpha, \beta}$ generated by the 6-element set
$Y = \{a,b,c,x,y,z \}$ subject to the 26-element set of relations
 $$
 S = bY \cup Ya \cup \{x,y,z\}^2 \cup \{ a(x+y-z) , (x-y)b, (x-z)b, xc-\alpha cx, yb-\beta cy, zc-cz \}.
 $$ 
Then $$
h_A(n+3) = \left\{
\begin{array}{ll}
11, & \alpha^n+\beta^n =1,\\
10, & \alpha^n+\beta^n \ne 1,
\end{array}
\right.
$$ 
for each $n\ge 0$.
In particular, the Fermat equation $ \alpha^n+\beta^n =1$ has no nonzero solution in the field $\KK$ for each $n\ge 3$
if and only if for each $A = A_{\alpha, \beta}$ we have $h_A(i) =10$ for all $i\ge 6$.
\end{prop}

\begin{proof}
Let $m = 3$ and $X = \{x,y,z\}$. Let $L = \KK (x+y+z)$, $R = \KK (x-y, x+z)$, and let $\sigma$ be the automorphism of $V$ 
sending $x$ to $\alpha x$, $y$ to $\beta y$, and $z$ to itself. Then $A_{\alpha,\beta} =  A (V,L,R, \sigma) $.
Note that we have $\sigma^n(x+y+z) = \alpha^n x+\beta^n y+z$, so that $\sigma^n(L)\subset R $ iff $\alpha^n+\beta^n =1$. Then the proposition follows from Corollary~\ref{cor:hA_constr}.
\end{proof}

Up to a linear change of variables, the algebra $A_{\alpha, \beta}$ coincides with the algebra from Corollary~\ref{cor:reccurent} with the linear recurrence $a_n=\alpha^n+\beta^n-1$. Note that the recurrent relation is $a_{n} = \alpha^{-1} \beta^{-1} \left( 
(\alpha+\beta+\alpha \beta) a_{n-1} - (\alpha+\beta+1)a_{n-2}+a_{n-3}
\right)$.

The next theorem shows that if $\KK$ contains the algebraic closure of $\Q$ or at least all cyclotomic fields then 
the both period and the the initial non-periodic segment of the Hilbert function of a quadratic $\KK$-algebra 
with fixed number of generators could be arbitrary large. 

\begin{theorem}
\label{th:periods_and_segments}

(a) If the field $\KK$ is infinite, then for each $g \ge 5 $ and each $d\ge 1$ there exists a $g$-generated quadratic $\KK$-algebra of linear growth with periodic Hilbert function $h_A(n)$ such that the initial non-periodic segment of the sequence $h_A$ has length $d$.

(b) If, in addition, $\KK$ contains all primitive roots of unity, then for each $g\ge 5$, $d\ge 5$ and $T\ge 1$ such that $T$ does not divide $d-4$ there exists a $g$-generated quadratic $\KK$-algebra of linear growth those Hilbert function has period  exactly $T$ and the length of the  initial non-periodic segment $d$.
\end{theorem}

The next lemma is a particular case of Proposition~\ref{prop:more_general_reccurence}.

\begin{lemma}[Arbitrary long non-periodic segment and period]
\label{lem:long_init_segment}
Suppose that the field $\KK$ is infinite. Let $\alpha\in \KK^{\times}$ be an element of order 
$\omega \in \Z
\cup\{\infty \}$.
Let $V = \KK^2$, let $\sigma $ be defined by the matrix $\left(
\begin{array}{cc}
\alpha & 0\\
0 &1
\end{array}
\right)$, and let $R$ and $L$ be one-dimensional vector spaces generated by the vectors $(1,-\alpha^\rho)^T$ and $(1,1)^T$, where $\rho > 0$. 
Then the algebra $A = A(V,R,L,\sigma)$ has 5 generators, 17 relations, and the Hilbert function 
$$
h_A(n+3) = \left\{
\begin{array}{ll}
9, & n=\rho \mbox{ or }  \omega|n ,\\
8, & \mbox{ otherwise}
\end{array}
\right.
$$
(where $n \ge 0)$.
\end{lemma}

\begin{proof}[Proof of Theorem~\ref{th:periods_and_segments}]
%
Let $\rho =d-4$. Since the field $\KK$ is infinite, for each $d\ge 5$ there exists an element $\alpha \in \KK^\times$ of order $\omega$ such that $\omega \not | \rho$  (for example, one can take any element of order  $\omega >\rho$).
Let $A$ be the algebra  from Lemma~\ref{lem:long_init_segment}. 
Then  the sequence $h_A$ has period either $ \omega $ (if $\omega<\infty$) or $1$ (if $\omega = \infty$)
 and the length of the initial non-periodic segment exactly $d$.
If, in addition, the field $\KK^\times$ contains all primitive roots of unity, one can assume here that $\omega = T$. 
Then the algebra $A$ satisfies the conditions of Theorem~\ref{th:periods_and_segments} with $g=5$.

Suppose that $g\ge 6$.
For $t\ge 1$, let $B_t = \bigoplus_t \KK[x]$ be the direct sum  with common unit of $t$ copies of the one-variable polynomial algebra.
Then the algebra $C = A\oplus B_{g-5}$ satisfies all the conditions of Theorem~\ref{th:periods_and_segments} with the Hilbert function $h_C(n) = h_A(n) +g-5$.
\end{proof}

\begin{rema}
If $T$  divides $d-4$, then the algebra from the conclusion of
 Theorem~\ref{th:periods_and_segments}b exists for each $g\ge 7$ provided that the field $\KK$ contains an element $\beta$ of infinite order. 

Indeed, 
let $A_1 = A(V_1, R_1, L_1, \sigma_1)$ be the algebra from Lemma~\ref{lem:long_init_segment} with $\rho = d-4$ and $\alpha =\beta$
and $A_2 = A(V_2, R_2, L_2, \sigma_2)$ be the algebra from Lemma~\ref{lem:long_init_segment} with $\rho = d-4$ and some $\alpha$ of order $T$.  
 Let $V = V_1 \oplus V_2 $,
 $R = R_1 \oplus R_2 $, $L = L_1 \oplus L_2 $, $\sigma = \sigma_1 \oplus \sigma_2$, and $A = A(V,R,L,\sigma)$. 
 It follows now from Corollary~\ref{cor:hA_constr} that $h_A(n+3) = 16 + c_n$ for all $n\ge 0$, where 
 $c_n = \dim (R\cap \sigma^n L) = \dim (R_1\cap \sigma_1^n L_1) +\dim  (R_2\cap \sigma_2^n L_2) $, that is,
 $$
    c_n  =  \left\{ 
\begin{array}{ll}
2, & n=0 \mbox{ or } n = d-4,\\
1, & T |n \mbox{ and } n\ne 0, d-4, \\
0, & \mbox{ otherwise.} 
\end{array}
\right.
$$
We see that the Hilbert function of the $g$-generated algebra $C = A\oplus B_{g-7}$ 
has the form $h_C(n+3) = 16+(g-7)+c_n$, so that $C$ is as required. 
\end{rema}

The first and the simplest example of a recurrent sequence with non-periodic set of zeroes has been constructed by Lech~\cite{lech}. If the field $\KK$ has characteristic zero and $x\in \KK$ is transcendental over a prime subfield 
then the sequence $a_n = (x+1)^{n}-x^{n}-1$ satisfies $a_n =0$ iff $n=p^m$ with $m\ge 0$. 
Derksen~\cite{derk} considered a more general class of examples. 
The next example of an algebra with non-periodic Hilbert function corresponds to Lech's linear recurrence by 
Proposition~\ref{prop:more_general_reccurence}. 

\begin{prop}[A non-periodic Hilbert function]
\label{prop:non-periodic}
Suppose that the field $\KK$ is a transcendental extension of a prime finite field  $\F_p$, and let $x\in \KK$ be a transcendental element over  $\F_p$.  
Let $V = \KK^3$, let $\sigma $ be defined by the matrix $\left(
\begin{array}{ccc}
x+1 & 0 & 0\\
0 &x & 0\\
0 & 0 & 1
\end{array}
\right)$, and let $L$ and $R$ be the spans of $(1,1,1)^T$ and $\{ (1,1,0)^T , (1,0,1)^T  \}$, respectively.
 Then the 6-generated  algebra $A = A(V,R,L,\sigma)$ has  non-periodic Hilbert function
$$
h_A(n+3) = \left\{
\begin{array}{ll}
11, & n=p^m \mbox{ for some } m \ge 0,\\
10, & \mbox{otherwise}
\end{array}
\right.
$$
(where $n \ge 0)$.
\end{prop}

\end{document}